\documentclass[12 pt]{article}
\usepackage{amsmath, amsthm, amssymb,enumerate,algorithm2e,tikz,caption}
\usepackage{tikz}
\usetikzlibrary{shapes,snakes,fit}
\usepackage{fullpage}
\usepackage{enumitem}
\usepackage{hyperref}
\usepackage{xcolor}

\title{Tree Matchings}
\author{Alexander Roberts\thanks{Mathematical Institute, University of Oxford, Andrew Wiles Building, Radcliffe Observatory Quarter, Woodstock Road, Oxford, United Kingdom. \newline  E-mail: \texttt{robertsa@maths.ox.ac.uk}.}}

\newtheoremstyle{case}{}{}{\normalfont}{}{\itshape}{:}{ }{}

\newtheorem{thm}{Theorem}[section]
\newtheorem{lem}[thm]{Lemma}
\newtheorem{prop}[thm]{Proposition}

\theoremstyle{definition}
\newtheorem{defn}[thm]{Definition}

\numberwithin{equation}{section}

\newtheoremstyle{case}{}{}{\normalfont}{}{\itshape}{\normalfont:}{ }{}

\theoremstyle{case}

\def\comment#1{}

\newcommand{\beq}{\begin{eqnarray*}}
\newcommand{\eeq}{\end{eqnarray*}}

\def\build#1_#2^#3{\mathrel{\mathop{\kern 0pt#1}\limits_{#2}^{#3}}}
\newcommand{\beqs}{\begin{eqnarray}}
\newcommand{\eeqs}{\end{eqnarray}}

\newcommand{\cC}{\mathcal{C}}

\numberwithin{equation}{section}

\tikzset{
    position/.style args={#1:#2 from #3}{
        at=(#3.#1), anchor=#1+180, shift=(#1:#2)
    }
}
\begin{document}
\maketitle
\begin{abstract}
An $(s,t)$-matching in a bipartite graph $G=(U,V,E)$ is a subset of the edges $F$ such that each component of $G[F]$ is a tree with at most $t$ edges and each vertex in $U$ has $s$ neighbours in $G[H]$. We give sharp conditions for a bipartite graph to contain an $(s,t)$-matching. As a special case, we prove a conjecture of Bonacina, Galesi, Huynh and Wollan \cite{CNF}.
\end{abstract}
\section{Introduction}
\tikzstyle{vertex} = [draw,shape=circle,node distance=80pt]
\tikzstyle{hex} = [draw,regular polygon, regular polygon sides = 6]
\tikzstyle{switch} = [fill,shape=circle,node distance=80pt]
\tikzstyle{vsquare} = [draw,shape=rectangle,node distance = 80pt]
\tikzstyle{edge} = [draw,opacity=0.1,line cap=round, line join=round, line width=30pt]

Let $G = (U,V,E)$ be a bipartite graph. A \it matching \rm from $U$ to $V$ is a subset $F$ of pairwise disjoint edges from $E$ such that each vertex from $U$ is incident to an edge in $F$. For $\alpha > 0$ we will say that $G$ satisfies the \it $\alpha$-neighbourhood condition \rm if $|\Gamma(S)| \ge \alpha |S|$ for each $S \subset U$. A fundamental result in matching theory is Hall's Theorem.

\begin{thm}[Hall's Theorem {\cite{Hall1}}]
Let $G=(U,V,E)$ be a bipartite graph, then $G$ has a matching from $U$ to $V$ iff $G$ satisfies the $1$-neighbourhood condition.
\end{thm}

It follows easily from Hall's Theorem that if $G$ satisfies the $h$-neighbourhood condition then $G$ has an $(h,h)$-matching, or in other words a collection of vertex disjoint stars $K_{1,h}$ centred on the vertices of $U$. But what happens if $G$ does not quite satisfy the $h$-neighbourhood condition? $G$ no longer has an $h$-matching, but perhaps we can choose $h$ edges incident with each vertex of $U$ so that the resulting graph has only small components. 

\begin{defn}
Let $t \ge s$ be positive integers and $G = (U,V,E)$ be a bipartite graph. An $(s,t)$-matching is a subset $F$ of $E$ such that in $H = (U,V,F)$, each component is a tree with at most $t$ edges, and $d_H(u) = s$ for each $u \in U$.
\end{defn}

A special case of this question was raised in a paper of Bonacina, Galesi, Huynh and Wollan \cite{CNF}. That paper considered a covering game on a bipartite graph. It turned out that which player wins is strongly linked to the existence of a $(2,4)$-matching in $G$. Bonacina, Galesi, Huynh and Wollan showed that for $\epsilon < \frac{1}{23}$ the $(2-\epsilon)$-neighbourhood condition is sufficient for the existence of a $(2,4)$-matching in a bipartite graph $G$ with maximal left degree at most $3$. They conjectured that the result should hold for $\epsilon = \frac{1}{3}$. In this paper, we prove their conjecture as a special case of a much more general result.

We will give sufficient neighbourhood conditions for the existence of $(h,hk)$-matchings for general $h,k$.

\begin{thm}\label{summary}
Let $k \ge 1$ and $h \ge 2$ be positive integers and let $G = (U,V,E)$ be a bipartite graph. Suppose that for all $S \subset U$,
	\beqs
		|\Gamma(S)| \ge \biggr(h-1+\frac{1}{\lceil k/h \rceil}\biggl)|S|. \nonumber
	\eeqs
Then $G$ has an $(h,hk)$-matching.
\end{thm}

We will actually prove a stronger result which conditions on the maximum left degree of the bipartite graph.

\begin{thm}\label{main}
Let $k \ge 1$ and $d > h \ge 2$ be positive integers and let $G = (U,V,E)$ be a bipartite graph. Suppose that $d(u) \le d$ for all $u \in U$ and, for all $S \subset U$,
	\beqs
		|\Gamma(S)| \ge \biggr(h-1+\frac{d-h+1}{k+1+(d-h-1)\lceil k/h \rceil}\biggl)|S|. \nonumber
	\eeqs
Then $G$ has an $(h,hk)$-matching.
\end{thm}

By taking the limit as $d$ tends to infinity, one can see that Theorem \ref{summary} follows directly from Theorem \ref{main}. Taking $h,k=2$ and $d=3$, we also see that Theorem \ref{main} proves the conjecture of Bonacina, Galesi, Huynh and Wollan \cite{CNF} mentioned above.

Showing these $\alpha$-bounds is a little tricky; unlike the case of Hall's Theorem, the conditions in Theorem \ref{main} are sufficient but not necessary (for example, $K_{2,3}$ contains a $(2,4)$-matching but does not satisfy the $\frac{5}{3}$-neighbourhood condition); and it is necessary to provide an infinite family of examples for $\alpha$ increasing to the relevant threshold as the example increases in size. Bonacina, Galesi, Huynh and Wollan \cite{CNF} provide a example to show that for any $\alpha < \frac{5}{3}$, there exists a bipartite graph $G$ with maximal left degree $3$ which satisfies the $\alpha$-neighbourhood condition but does not contain a $(2,4)$-matching. We will modify this particular family of examples to give examples for all values of $d$, $k$ and $h$. These examples show that the sufficient neighbourhood conditions given in Theorem \ref{main} are in fact optimal.

\begin{prop}\label{counter}
Let $k \ge 2$ and $d > h \ge 1$ be positive numbers and $\alpha < h-1+\frac{d-h+1}{k+1+(d-h-1)\lceil k/h \rceil}$. Then there exists a bipartite graph $G$ with maximum left degree at most $d$ which satisfies the $\alpha$-neighbourhood condition but does not contain an $(h,hk)$-matching.
\end{prop}

The paper is organised as follows. In Section \ref{prelim} we prove some preliminary results regarding bipartite graphs that satisfy a neighbourhood condition which is no longer satisfied upon the deletion of any edge. In Section \ref{secmain} we prove Theorem \ref{main}. In Section \ref{tight} we expose the examples which prove Proposition \ref{counter} and so demonstrate the bounds given in Theorem \ref{main} are tight. Finally, in Section \ref{starsect}, we consider a related covering problem.

\section{Preliminary results}\label{prelim}
As stated before, we will prove Theorem \ref{main} by induction on the number of edges in the graph. For a bipartite graph $G = (U,V,E)$, we will call an edge, $e \in E$, \it $\alpha$-redundant \rm if $G-e$ satisfies the $\alpha$-neighbourhood condition. In other words, an edge is redundant if it is not necessary for the satisfaction of neighbourhood constraints. This section will show that if a connected bipartite graph satisfies the $\alpha$-neighbourhood condition and has no redundant edge, then (subject to a couple of other restraints) it must be a tree. We will start the section by introducing some notation which will be used throughout the paper.

Let $G=(U,V,E)$ be a bipartite graph. For $A \subset U$ and $\alpha > 0$, let $h(A,\alpha) = |\Gamma(A)| - \alpha|A|$ (so $G=(U,V,E)$ satisfies the $\alpha$-neighbourhood condition if and only if $h(A,\alpha) \ge 0$ for each $A \subset U$). Then for $uv \in E$, $u \in U, v \in V$, let $F_{uv} = \{A \subset U : u \in A, v \notin \Gamma(A \setminus u) \}$ and $G_{uv} = \{A \subset U\setminus u : v \in \Gamma(A)\}$. Then we define functions $f$ and $g$:
	\beqs
		f(uv,\alpha) &=& \min_{A \in F_{uv}}h(A,\alpha) \nonumber \\
		g(uv,\alpha) &=& \min_{A \in G_{uv}}h(A,\alpha) \nonumber
	\eeqs
where we put $g(uv, \alpha) = 1$ if $d(v) = 1$ (and so $G_{uv} = \emptyset$). We will drop the $\alpha$ when obvious or when it's value is inconsequential.

$f(uv)$ can be thought of as a measure of how redundant an edge $uv$ is and $g(uv)$ can be thought of as a measure of how redundant the vertex $v$ is to the graph $G-u$; in other words, how little is it required by other vertices. For a graph satisfying the $\alpha$-neighbourhood condition it is clear that $f(uv,\alpha),g(uv,\alpha) \ge 0$ for each $uv \in E$. The next proposition analyses some properties of $f$ and $g$ on a graph satisfying the $\alpha$-neighbourhood condition.

\begin{prop}\label{properties}
Let $\alpha > 0$ and $G = (U,V,E)$ be a bipartite graph which satisfies the $\alpha$-neighbourhood condition.
\begin{itemize}
\item[(i)] An edge $uv$ is $\alpha$-redundant if and only if $f(uv,\alpha) \ge 1$.
\item[(ii)] For $v \in V$, $u,w \in \Gamma(v)$, $F_{wv} \subset G{uv}$ and so $g(uv,\alpha) \le f(wv,\alpha)$.
\item[(iii)] Suppose further that $G$ does not contain a redundant edge. For an edge $uv \in E$, $g(uv,\alpha) \le 1$ with equality if and only if $d(v) = 1$.
\end{itemize}
\end{prop}

\begin{proof}
Let $\alpha > 0$ and $G = (U,V,E)$ be a bipartite graph which satisfies the $\alpha$-neighbourhood condition. Suppose $uv \in E$ and let $H = G - uv$.

\begin{itemize}
\item[(i)] First suppose that $f(uv,\alpha) < 1$ and let $A \in F_{uv}$ be such that $h_G(A) = f(uv,\alpha)$. Note that by definition of $F_{uv}$, $\Gamma_H(A) = \Gamma_G(A) \setminus v$ and so $h_H(A) = h_G(A) - 1 < 0$. We can then conclude that $uv$ is not redundant since $H$ does not satisfy the $\alpha$-neighbourhood condition. 

Now suppose that $uv$ is not redundant. By definition, there must be some $A \subset U$ such that $h_H(A) < 0$. Note that since $G$ satisfies the $\alpha$-neighbourhood condition, such a subset $A$ must contain $u$ and that $v \notin \Gamma(A \setminus u)$. It is then clear that $A \in F_{uv}$ and so $f(uv, \alpha) \le h_G(A) = h_H(A) + 1 < 1$.

\item[(ii)] Note that if $S \in F_{wv}$, then $S \subset U\setminus u$ and $v \in \Gamma(S)$ and so $S \in G_{uv}$. It follows that if $u,w \in \Gamma(v)$, then $F_{wv} \subset G_{uv}$ and so $g(uv,\alpha) \le f(wv,\alpha)$.

\item[(iii)] Recall that if $d(v) = 1$, then $g(uv,\alpha) = 1$ by definition. So suppose that $d(v) \ge 2$ and pick some $w \in \Gamma(v) \setminus u$. Then since $vw$ is not a redundant edge, $f(wv,\alpha) < 1$ and the results follows from (ii).
\end{itemize}
\end{proof}

The following lemma considers the effect of applying $h$ to a union of two sets and will be used extensively in the remainder of the section.

\begin{lem}\label{union}
Let $G = (U,V,E)$ be a bipartite graph and fix some $\alpha>0$. Then for $A,B \subset U$,
	\beqs
		h(A \cup B) = h(A) + h(B) - h(A \cap B) - (|\Gamma(A)\cap \Gamma(B)| - |\Gamma(A \cap B)|). \nonumber
	\eeqs
\end{lem}

\begin{proof}
Note that $|\Gamma(A\cup B)| = |\Gamma(A)| + |\Gamma(B)| - |\Gamma(A) \cap \Gamma(B)|$ and $|A \cup B| = |A| + |B| - |A\cap B|$, so
	\beqs
		h(A \cup B) &=& |\Gamma(A \cup B)| - \alpha|A \cup B| \nonumber \\
		&=& |\Gamma(A)| + |\Gamma(B)| - |\Gamma(A) \cap \Gamma(B)| - \alpha(|A| + |B| - |A\cap B|) \nonumber \\
		&=& h(A) + h(B) - (|\Gamma(A) \cap \Gamma(B)| - \alpha |A \cap B|) \nonumber \\
		&=& h(A) + h(B) - (|\Gamma(A) \cap \Gamma(B)| - |\Gamma(A \cap B)|) - (|\Gamma(A \cap B)| - \alpha |A \cap B|) \nonumber \\
		&=& h(A) + h(B) - h(A \cap B) - (|\Gamma(A)\cap \Gamma(B)| - |\Gamma(A \cap B)|). \nonumber
	\eeqs
\end{proof}

We are now in a position to show that, under additional constraints regarding $h$, a bipartite graph satisfying the $\alpha$-neighbourhood condition with no redundant edges, must be a tree.

\begin{lem}\label{tree}
Let $\alpha > 0$ and let $G = (U,V,E)$ be a bipartite graph with no isolated vertices. Suppose that $h(U,\alpha) \in [0,1)$,  $h(S,\alpha) > h(U,\alpha)$ for each $S \subsetneq U$, and that $G$ contains no $\alpha$-redundant edges. Then $G$ is a tree.
\end{lem}

\begin{proof}
Since $\alpha$ is fixed, we will write $h(S),f(uv)$ and $g(uv)$ in place of $h(S,\alpha), f(uv,\alpha)$ and $g(uv,\alpha)$ respectively.

First suppose that $G$ is not connected. Let $A \cup \Gamma(A)$ be the vertex set of a component of $G$ with $A \subset U$ and let $B = U \setminus A$. As $\Gamma(A)$ and $\Gamma(B)$ are disjoint we have $h(U) = h(A) + h(B)$ by Lemma \ref{union}. Note that $h(B) > h(U) > 0$ by assumption and so $h(U) > h(A) +h(U)$. We have arrived at a contradiction since $h(A) > 0$. So $G$ is connected.

Now suppose that $G$ contains a cycle. Choose an edge $uv$ that belongs to a cycle which has 
$g(uv)$ as small as possible. Choose $A \in F_{uv}$ and $B \in G_{uv}$ such that $h(A) = f(uv)$ and $h(B) = g(uv)$.

Suppose that $G[A \cup \Gamma(A)]$ is disconnected and $J \subset A$ is such that $G[J \cup \Gamma(J)]$ is a component of $G[A \cup \Gamma(A)]$. $\Gamma(J)$ and $\Gamma(A \setminus J)$ are disjoint, and so $h(A) = h(J) + h(J \setminus A)$. Note that since both $J$ and $J \setminus A$ are non-trivial subsets of $U$, we have $\min\{h(J), h(J \setminus A)\} > 0$ and so $h(A) > \max\{h(J), h(J \setminus A)\}$. On the other hand, assuming without loss of generality that $u \in J$, we have that $J \in F_{uv}$ and so $h(J) \ge h(A)$, a contradiction. Similarly, suppose that $G[B \cup \Gamma(B)]$ is disconnected and that $K \subset B$ is such that $G[K \cup \Gamma(K)]$ is a component of $G[B \cup \Gamma(B)]$. $\Gamma(K)$ and $\Gamma(B \setminus K)$ are disjoint and so $h(B) = h(K) + h(B \setminus K)$. Since $K$ and $B \setminus K$ are non-trivial subset of $U$, we have $\min\{h(K), h(B \setminus K)\} > 0$ and so $\max\{h(K), h(B \setminus K)\} < h(B)$. On the other hand, assuming without loss of generality that $v \in \Gamma(K)$, we see that $K \in G_{uv}$ and so $h(K) \ge h(B)$, a contradiction. Thus we may assume that both $G[A \cup \Gamma(A)]$ and $G[B \cup \Gamma(B)]$ are connected.

Letting $C = A \cup B$ and $D = A \cap B$, an application of Lemma \ref{union} gives
	\beqs
		h(C) &=& h(A) + h(B) - h(D) - (|\Gamma(A)\cap \Gamma(B)| - |\Gamma(A \cap B)|) \nonumber \\
		&=& f(uv) + g(uv) - h(D) - (|\Gamma(A)\cap \Gamma(B)| - |\Gamma(A \cap B)|). \label{cycle1}
	\eeqs
Note that $G$ satisfies the $\alpha$-neighbourhood condition, has no redundant edges, and $d(v) \ge 2$. The conditions for Proposition \ref{properties} are therefore satisfied and so $f(uv),g(uv) <1$. If $|\Gamma(A)\cap \Gamma(B)| - |\Gamma(A \cap B)| \ge 2$, then $h(C) < 0$ and we arrive at a contradiction. Therefore $|\Gamma(A)\cap \Gamma(B)| - |\Gamma(A \cap B)| \le 1$ and so 
	\beqs
	\Gamma(A)\cap \Gamma(B) = \Gamma(A \cap B) \cup \{v\}, \label{plusv}
	\eeqs	
as $v \in \Gamma(A) \cap \Gamma(B)$ but $v \notin \Gamma(A \cap B)$. In particular, $|\Gamma(A)\cap \Gamma(B)| - |\Gamma(A \cap B)| = 1$. Putting this into \eqref{cycle1} gives
	\beqs
		h(C) = f(uv) + g(uv) - h(D) -1. \label{cycle2}
	\eeqs

Now suppose that $D \neq \emptyset$ and choose some vertex $w \in D$. Since $G[A \cup \Gamma(A)]$ and $G[B \cup \Gamma(B)]$ are both connected, there exists a $v-w$ path, $v=a_1 \cdots a_r=w$ in $G[A \cup \Gamma(A)]$ and a $v-w$ path $v=b_1\cdots b_t = w$ in $G[B \cup \Gamma(B)]$. Note that $A \in F_{uv}$ and so $A \cap \Gamma(v) = \{u\}$, forcing $a_2 = u$. This means that $a_2 \neq b_2$ and so the two $v-w$ paths are distinct. Let $i>1$ be minimal such that $a_i = b_j$ for some $j >1$ and fix $j$ minimal with $b_j = a_i$. Note then that $a_1,a_2,\ldots,a_i,b_2,b_3 \ldots b_{j-1}$ are distinct vertices and so $a_1 a_2 \cdots a_i b_{j-1}b_{j-2}\cdots b_2$ is a cycle in $G[C \cup \Gamma(C)]$.

Note that either $a_i \in A \cap B = D$ or $a_i \in \Gamma(A) \cap \Gamma(B)$. If the latter is the case, then $a_i \in \Gamma(D)$ since $a_i \neq v$, and, by \eqref{plusv}, $\Gamma(A) \cap \Gamma(B) = \Gamma(D) \cup \{v\}$. In either case, $a_i \in D \cup \Gamma(D)$. Then since $a_2 = u \notin D \cup \Gamma(D)$, there must be some $s < i$ such that $a_s \notin D \cup \Gamma(D)$, but $a_{s+1} \in D \cup \Gamma(D)$. If $a_{s+1} \in D$, then $a_s \in \Gamma(D)$, which gives a contradiction. So $a_s \in U \setminus D$ and $a_{s+1} \in \Gamma(D)$. This means that $D \in G_{a_s a_{s+1}}$ and so $h(D) \ge g(a_s a_{s+1})$. Finally, since $a_s a_{s+1}$ is an edge in a cycle and we have chosen $uv$ to minimise $g(uv)$, it must be that $g(a_s a_{s+1}) \ge g(uv)$. If we put this inequality into \eqref{cycle2} we get
	\beqs
		h(C) &\le& f(uv) + g(uv) - g(uv) -1 \nonumber \\
		&=& f(uv) - 1.
	\eeqs
Since $f(uv) < 1$, we have that $h(C) < 0$, a contradiction. It must therefore be the case that $D = \emptyset$.

Now suppose $uv$ is in some cycle $F = u_1v_1 \ldots u_m v_m$ with $(u_1,v_m) = (u,v)$ and let $Q = \{u_1,\ldots,u_m\}$. Note that $u_m \notin A$ and $u_1 \notin B$. So if $Q \subset C = A \cup B$, then there exists some $j \neq m$ with $u_j \in A$, $u_{j+1} \in B$. It is then the case $v_j \in \Gamma(A) \cap \Gamma(B) = \{v\}$, which gives a contradiction. So there exists some $j$ such that $u_{j+1} \notin C$ but $u_j \in C$. This means that $C \in G_{u_{j+1}v_j}$ and so $h(C) \ge g(u_{j+1}v_j) \ge g(uv)$. Since $D = \emptyset$ and $\Gamma(A) \cap \Gamma(B) = \{v\}$, \eqref{cycle2} gives
	\beqs
		h(C) &=& f(uv) + g(uv) - 0 - 1 \nonumber \\
		&<& g(uv). \nonumber
	\eeqs
But this again gives a contradiction since then $g(u_{j+1}v_j) < g(uv)$. It follows that $G$ cannot contain a cycle and so must be a tree.
\end{proof}

\section{Proof of Theorem \ref{main}}\label{secmain}
We now come to proving Theorem \ref{main}. We will prove this by considering an edge-minimal counterexample and arriving at a contradiction. The first half of the proof will show that this counterexample must be a tree and thus acyclic. The second half will show that the counterexample, at the same time, must contain a cycle.

\begin{proof}[Proof of Theorem \ref{main}]
Let $k \ge 1$ and $d > h \ge 2$ be positive integers, fix $r = d-h$ and let
	\beqs
		\alpha = h-1 + \frac{r+1}{k+1 + (r-1)\lceil k/h \rceil}. \nonumber
	\eeqs
Since $\alpha$ is fixed, we will write $h(S),f(uv)$ and $g(uv)$ in place of $h(S,\alpha), f(uv,\alpha)$ and $g(uv,\alpha)$ respectively.

Suppose that there exists a bipartite graph with maximum left degree at most $h+r$ which satisfies the $\alpha$-neighbourhood condition but doesn't have an $(h,hk)$-matching. Let $G =(U,V,E)$ be edge-minimal with these properties and assume without loss of generality that there are no isolated vertices in $G$. Note that $G$ has minimum left degree at least $h$ since $h(\{u\}) \ge 0$ for each $u \in U$. Suppose that there exists some non-empty $A \subsetneq U$ such that $h(A) \le h(U)$ and suppose that $A$ is such a set with minimal $h(A)$. It is clear that the subgraph $H = G[A\cup \Gamma(A)]$ satisfies the $\alpha$-neighbourhood condition and has fewer edges than $G$ and so by assumption must have an $(h,hk)$-matching. If $B \subset U \setminus A$, then
	\beqs
		|\Gamma_{G-(A \cup \Gamma(A))}(B)| -\alpha|B| &\ge& (|\Gamma_G(B\cup A)|-\alpha|B\cup A|) - (|\Gamma_G(A)|-\alpha|A|) \nonumber \\
		&=& h(A\cup B) - h(A) \nonumber \\
		&\ge& 0. \nonumber
	\eeqs
Thus $G-(A\cup \Gamma(A))$ satisfies the $\alpha$-neighbourhood condition and so by assumption must contain an $(h,hk)$-matching. The two $(h,hk)$-matchings are vertex-disjoint and so their union is an $(h,hk)$-matching in $G$. This gives a contradiction and so we must have $h(A) > h(U)$ for each non-empty $A \subsetneq U$.

If there exists an $\alpha$-redundant edge $uv$ in $G$ (equivalently $f(uv)\ge 1$), then we may simply delete it to find a counter-example with fewer edges which contradicts the minimality of $G$. So $f(uv) \in [0,1)$ for each edge $uv$ in $E$.  If we pick some edge $uv$, we have that $f(uv) < 1$ and so there must be some $A \in F_{uv}$ with $h(A) < 1$. Then since $h(U) \le h(A)$ for non-empty $A \subset U$, it must be the case that $h(U) < 1$. We have now shown that $G$ satisfies all the conditions for Lemma \ref{tree} and so $G$ must be a tree.

For each positive integer $i$, let $U_i = \{u \in U : d(u) = i\}$ and $V_i = \{v \in V: d(v) = i\}$. Then let $F = \{u \in U_h : |\Gamma(u) \cap V_1| = h-1\}$ and $Z = U_h \setminus F$. Suppose that $C \cup \Gamma(C)$ is a component of $G[F \cup \Gamma(F)]$. For each $u \in C$, let $X_u = \Gamma(u) \cap V_1$ and $Y_u = \Gamma(u) \setminus V_1$. Note that by pruning the leaves of $G[C \cup \Gamma(C)]$ contained in $V_1$, we get $G[C \cup \bigcup_{u \in C} Y_u]$. We can then see that $G[C \cup \bigcup_{u \in C} Y_u]$ is a tree and so $e(G[C \cup \bigcup_{u \in C} Y_u]) = |C \cup \bigcup_{u \in C} Y_u| - 1 = |C| + |\bigcup_{u \in C} Y_u| - 1$. On the other hand $|Y_u| = 1$ for each $u \in C$ and so $e(G[C \cup \bigcup_{u \in C} Y_u]) = |C|$. Comparing these two expressions we see that $|\bigcup_{u \in C} Y_u| = 1$. 

$G$ is connected and so $\Gamma(C)$ and $\Gamma(U \setminus C)$ must have a non-empty intersection. Note however that $\Gamma(U \setminus C) \cap \Gamma(C) \subset \bigcup_{u \in C} Y_u$ since all vertices in $\bigcup_{u \in C}X_u$ are leaves. Therefore each component of $G[F \cup \Gamma(F)]$ has exactly one vertex in $\Gamma(U \setminus F)$. In this case, we will say that $F$ satisfies the \em critical link property\em.

The following algorithm adds vertices from $Z$ to $F$ as long as it is possible to do so under the constraint that $F$ must always satisfy the critical link property.

\begin{algorithm}
\SetKw{KwFn}{Initialization}
	\KwFn{Set $\eta = \emptyset$} \;
	\While{$Z \neq \emptyset$}{
		Pick $u \in Z$\;
  		\eIf{$|\Gamma(u) \cap \Gamma(U\setminus(F \cup \{u\}))| = 1$}{
   			set $Z = \eta \cup (Z \setminus \{u\})$, $F = F \cup \{u\}$ and $\eta = \emptyset$ \;
   		}{
		Set $Z = Z \setminus \{u\}$ and $\eta = \eta \cup \{u\}$ \;
  		}
 	}
\end{algorithm}

We claim that after each iteration of the loop, $F$ still satisfies the critical link property. This is true initially and can only be changed in the loop, if we add a vertex to $F$. Suppose $u$ is added to $F$ at a certain stage and let $C \cup \Gamma(C)$ be the component of $(F \cup \{u\}) \cup \Gamma(F \cup \{u\})$ containing $u$. Note that all other components of $F$ will remain unchanged and so we only have to consider $C \cup \Gamma(C)$. Let $B = C \setminus \{u\}$ and note that $B \cup \Gamma(B)$ is the collection of components which are joined together by the addition of $u$ to $F$. Let $R = \Gamma(B) \cap \Gamma(U \setminus B)$ be the set of vertices in $V$ which connect the components of $B \cup \Gamma(B)$ to the rest of $G$ and note that by assumption $R$ must be a subset of the neighbourhood of $u$. Note that since we have added $u$ to $F$, it must be the case that $|\Gamma(u) \cap \Gamma(U\setminus(F \cup \{u\}))| = 1$. Further note that $\Gamma(C) \cap \Gamma(U\setminus(F \cup \{u\}))$ is a subset of $\Gamma(u)$ by the above argument and so $|\Gamma(C) \cap \Gamma(U \setminus C)| = 1$ as required.

So let us suppose we have augmented $F$ as far as we can by running the algorithm described above (so we have a subset $F \subset U_h$ such that $G[F \cup \Gamma(F)]$ is a forest which satisfies the critical link property and further that we cannot maintain this property if we add any vertex from $U_h$). Note that since each component of $G[F \cup \Gamma(F)]$ has exactly one vertex in the neighbourhood of $U \setminus F$, we know that $U \setminus F$ cannot be the empty set. So let $W = \{v \in V : |\Gamma(v) \setminus F| \ge 2\}$ be the subset of $V$ with at least two neighbours in $U \setminus F$. We claim that each $u \in U \setminus F$ has at least two neighbours in $W$ which will in turn mean that $G[(U \setminus F)\cup W]$ is a subgraph of $G$ with minimum degree at least $2$. We will then have arrived at a contradiction since this subgraph of the tree $G$ must then contain a cycle.

So pick $u \in U \setminus F$ and first suppose that $d(u) = h$. Since we have not added $u$ to $F$ whilst running the algorithm, either $\Gamma(u) \cap \Gamma(U \setminus (F \setminus \{u\})) = \emptyset$ or  $|\Gamma(u) \cap \Gamma(U \setminus (F \cup \{u\})| \ge 2$. In the latter case, note that $\Gamma(u) \cap \Gamma(U \setminus (F \cup \{u\}) \subset W$ and so $|\Gamma(u) \cap W| \ge 2$. In the former case, let $F^{+} = F \cup \{u\}$ and consider the component, $\cC = Q \cup \Gamma(Q)$, of $G[F^{+} \cup \Gamma(F^{+})]$ with $u \in Q \subset F^{+}$. Note that $\Gamma(Q) \cap \Gamma(U \setminus F^{+}) = \Gamma(u) \cap \Gamma(F^{+})$ since $u$ must be a neighbour of each vertex in $\Gamma(Q \setminus \{u\})\cap \Gamma(U \setminus (Q \setminus \{u\})$ and so $Q \cup \Gamma(Q)$ must be disconnected from the rest of the graph. Since $G$ is connected, it must then be the case that $Q=U$. Recall that $G$ is a bipartite tree. Counting edges two ways, we see that $h|U| = |U|+|V|-1$ and so $|V| = (h-1+\frac{1}{|U|})|U|$. On the other hand, recall that $G$ satisfies the $\alpha$-neighbourhood condition and so $|V| \ge \alpha |U|$. This in turn forces $(h-1+\frac{1}{|U|}) \ge \alpha$. We can then bound the size of $U$:
	\beqs
		|U| &\le& \frac{k+1+(r-1)\lceil k/h \rceil}{r+1} \nonumber \\
		&<& \frac{k+1}{2} + \left\lceil \frac{k}{h} \right\rceil \nonumber \\
		&\le& \frac{k+1}{2}+ \frac{k+1}{2}. \nonumber
	\eeqs
It must therfore be the case that $|U| \le k$. Now we have a contradiction since $G$ is already an $(h,hk)$-matching. Therefore any vertex $u \in U \setminus F$ with $d(u) = h$ has at least two neighbours in $W$.

Now suppose we have picked some $u \in U \setminus F$ with $d(u) \ge h+1$ but $|\Gamma(u) \cap W| < 2$. First consider what happens with $|\Gamma(u) \cap W| = 0$. Let, $\cC = Q \cup \Gamma(Q)$, be the component of $G[F \cup \{u\} \cup \Gamma(F \cup \{u\})]$ with $u \in Q \subset F$. As argued before, it must be the case that $\Gamma(Q\setminus \{u\}) \cap \Gamma(U \setminus (Q \setminus \{u\}))$ is a subset of $\Gamma(u)$. But note that $\Gamma(u) \cap \Gamma(U  \setminus Q) = \Gamma(u) \cap W = \emptyset$ and so $\cC$ must be the vertex set of a component in $G$. Since $G$ is connected, it must then be the case that $U = Q$. As in the case when $d(u) = h$, we can now count edges two ways to realise $h(|U|-1) + d(u) = |U|+|V|-1$ and so $|V| = (h-1 + \frac{d(u)+1-h}{|U|})|U|$. Again, we recall that $G$ satisfies the $\alpha$-neighbourhood condition and so $h-1 + \frac{d(u)+1-h}{|U|} \ge \alpha$. We can bound the size of $U$:
	\beqs
		|U| &\le& \frac{(d(u)+1-h)(k+1+(r-1) \lceil k/h \rceil)}{r+1}. \label{comp1}
	\eeqs

On the other hand, order the vertices of $\Gamma(u) = \{v_1,\ldots v_{d(u)}\}$ such that if $i < j$, then in $G - u$ the size of the component containing $v_i$ is at most the size of the component containing $v_j$ (alternatively, consider $G$ as a tree with root $u$ and order the branches by increasing size). If the $h$ shortest branches collectively contain at most $k-1$ vertices in $U$, then we can construct a matching, simply by cutting the edges $uv_{h+1},uv_{h+2},\ldots uv_{d(h)}$. Therefore the union of the smallest $h$ branches contain at least $k$ vertices from $U$. It must also be the case that all other branches contain at least $\lceil \frac{k}{h} \rceil$ left vertices. We now bound the size of $U$ by counting $u$, the vertices in the $h$ smallest branches, and the vertices in other branches:
	\beqs
		|U| \ge 1 + k + (d(u)-h)\left\lceil k/h \right\rceil. \label{comp21}
	\eeqs
	
After some algebra, we can reformulate \eqref{comp21} to get
	\beqs
		|U| &\ge& \frac{(d(u)+1-h)(k+1+(r-1) \lceil k/h \rceil)}{r+1} \nonumber \\
		&+& \frac{(r+h-d(u))(k+1 - 2 \lceil k/h \rceil) + (r+1)\lceil k/h \rceil}{r+1}. \label{comp2}
	\eeqs
Since $k+1 - 2 \lceil \frac{k}{h} \rceil \ge 0$, we see that the second term in \eqref{comp2} is positive and so our lower bound for $U$ here is strictly larger than the upper bound we have at \eqref{comp1}. So we have a contradiction and so it cannot be the case that $|\Gamma(u) \cap W| = 0$.

All that remains is to consider the case that $u \in U \setminus F$ is such that $d(u) \ge h+1$ and $|\Gamma(u) \cap W| = 1$. Suppose that $\Gamma(u) \cap W = \{w\}$ and let $Y = F \cup \{u\}$, $Z = \Gamma(Y) \setminus \{w\}$. Further let $A \cup B$ be the component of $G[Y \cup Z]$ such that $u \in A \subset Y$. Recall that $d(w) \ge 2$ by assumption and so there must exist some $u' \in U \setminus Y$. It is then the case that $A \neq U$, so that $h(A) > 0$ and $|B| > \alpha|A| -1$. Note that $G \setminus (A \cup B)$ will still satisfy the $\alpha$-neighbourhood condition and so by assumption, must contain an $(h,hk)$-matching. Let $H = G[A \cup B]$. Then if $H$ contains an $(h,hk)$-matching, it is independent of any $(h,hk)$-matching in $G \setminus (A \cup B)$ and their union is an $(h,hk)$-matching in $G$. Therefore $H$ does not contain an $(h,hk)$-matching. As in the previous case, we will now bound $|A|$ above and below to reach a contradiction. Firstly, since $H$ is a tree and we know the degrees of all the vertices in $A$, we can count the number of edges two ways to get that $h(|A|-1) + d(u)-1 = |A| + |B| -1$ and so $|B| = (h-1 + \frac{d(u)-h}{|A|})|A|$. Using the fact that $|B| > \alpha |A| - 1$, we get an upper bound for $|A|$:
	\beqs
		|A| < \frac{(d(u)-h+1)(k+1+(r-1) \lceil k/h \rceil)}{r+1}. \label{comp3}
	\eeqs
	
On the other hand, order the vertices of $\Gamma(u)\setminus \{w\} = \{v_1,\ldots v_{d(u)}\}$ such that if $i < j$, then in $H - u$ the size of the component containing $v_i$ is at most the size of the component containing $v_j$. Since we cannot have a matching, the $h$ smallest branches must collectively have at least $k$ vertices of $A$ in them and the $(d(u)-1-h)$ other branches must each contain at least $\lceil \frac{k}{h} \rceil$ vertices of $A$. Therefore we get a lower bound for $|A|$:
	\beqs
		|A| &\ge& 1 + k + (d(u)-1-h)\lceil k/h \rceil \nonumber \\
		&=& \frac{(r+1)(1 + k + (d(u)-1-h)\lceil k/h \rceil)}{r+1} \nonumber \\
		&=& \frac{(r+1)(k+1+(r-1) \lceil k/h \rceil) - (r+h-d(u))(r+1)\lceil k/h \rceil}{r+1} \nonumber \\
		&=& \frac{(d(u)-h+1)(k+1+(r-1) \lceil k/h \rceil)}{r+1} \nonumber \\
		&+& \frac{(r+h-d(u))(k+1 - 2\lceil k/h \rceil)}{r+1}. \label{comp4}
	\eeqs
Again, since $k+1 - 2\lceil \frac{k}{h} \rceil \ge 0$, we see that our lower bound for $|A|$ at \eqref{comp4} is at least the strict upper bound given at \eqref{comp3}. This is a contradiction and so it must be the case that $|\Gamma(u) \cap W| \ge 2$.

We have now shown that $G[(U \setminus F) \cup W]$ is a graph with at least one vertex and minimum degree at least $2$. Therefore $G$ must contain a cycle, contradicting that $G$ is a tree and so acyclic. So we can finally conclude that there can be no such counterexample and so the result holds.
\end{proof}

\section{Optimality}\label{tight}
In this section we give examples to show that the bounds given in Theorem \ref{main} are tight. Much of the material in this section builds on the work given in the paper of Bonacina, Galesi, Huynh and Wollan \cite{CNF} (this is very clear for the case $k=h$). For ease of notation, for a bipartite graph $G = (U,V,E)$ and a set $S \subset U$, we let $R_G(S) = \frac{|\Gamma(S)|}{|S|}$.

Rather than drawing the bipartite graph $G = (U,V,E)$, we will give pictorial representations of the hypergraph $H = (V,F)$ where $F = \{\Gamma(u) : u \in U\}$. So $H$ is the hypergraph on the right vertices $V$ of $G$, where each edge is the neighbourhood of a left vertex of $G$. Throughout the section, an ellipse represents the neighbourhood of a vertex in $U$ (i.e. a hyperedge of $H$), a small circle represents a single vertex in $V$, and a rectangle with a number $x$ inside represents a collection of $x$ vertices in $V$. For all the figures that follow, we will assume that parameters $a,b,h,q$ and $r$ are given. We give a toy example below where Figure \ref{exam2} is the hypergraph representation of Figure \ref{exam1}:

\begin{minipage}{0.45\textwidth}
\centering
\begin{tikzpicture}[thick,scale=0.6, every node/.style={scale=0.6}]
	\node[vertex] at (0,0) (x1) {} ;
	\node[vertex, position = -90:{3} from x1] (x2) {} ;
	\node[position = 0:{5} from x1] (ydum) {} ;
	\node[vertex, position = 90:{0} from ydum] (y2) {} ;
	\node[vertex, position = 90:{1} from ydum] (y1) {} ;
	\node[vertex, position = -90:{0.5} from ydum] (y3) {} ;
	\node[vertex, position = -90:{1.5} from ydum] (y4) {} ;
	\node[vertex, position = -90:{2.5} from ydum] (y5) {} ;
	\node[vertex, position = -90:{3.5} from ydum] (y6) {} ;
	\node[vertex, position = -90:{4.5} from ydum] (y7) {} ;
	\draw (x1) --(y1) ;
	\draw (x1) --(y2) ;
	\draw (x1) --(y3) ;
	\draw (x1) --(y4) ;
	\draw (x2) --(y4) ;
	\draw (x2) --(y5) ;
	\draw (x2) --(y6) ;
	\draw (x2) --(y7) ;
\end{tikzpicture}
\captionsetup{font=footnotesize}
\captionof{figure}{}\label{exam1}
\end{minipage}
\begin{minipage}{0.45\textwidth}
\centering
\begin{tikzpicture}[thick,scale=0.6, every node/.style={scale=0.6}]
	\node[vertex] at (0,0) (x1) {} ;
	\node[vsquare, position = -90:{1.2} from x1] (y1) {$2$} ;
	\node[position = 0:{1} from y1] (z1) {} ;
	\node[vertex, position = -90:{2.9} from x1] (x2) {} ;
	\node[vsquare, position = -90:{1.2} from x2] (y2) {$2$} ;
	\node[vertex, position = -90:{2.9} from x2] (x3) {} ;
	\node[ellipse, draw = black, fit= (x2) (y1) (x1)] {};
	\node[ellipse, draw = black, fit= (x2) (y2) (x3)] {};
\end{tikzpicture}
\captionsetup{font=footnotesize}
\captionof{figure}{}\label{exam2}
\end{minipage}

\medskip

To make the graph representations more digestible, we will use a hexagon so that Figures \ref{f1} and \ref{f2} represent the same graphs, which we will call $I_{q}$. Thus $I_q$ is a chain of $q$ hyperedges $A_1,\ldots, A_q$ each containing $h$ vertices such that $A_i$ and $A_{i+1}$ overlap in one vertex for each $i \le q-1$ and the $A_i$ are disjoint otherwise. Another way of thinking of $I_q$ is to start with a path consisting of $q$ left vertices and $q+1$ right vertices, adding another $h-2$ distinct leaf-neighbours to each left vertex in the path and then taking the hyperedge representation.
 
\begin{minipage}{0.45\textwidth}
\centering
\begin{tikzpicture}[thick,scale=0.6, every node/.style={scale=0.6}]
	\node[vertex, label = $v$] at (0,0) (x) {} ;
	\node[hex, position = -90:{1.7} from x] (y) {$q$} ;
	\node[vertex,position = -90:{4} from x, label = $w$] (z) {} ;
	\node[ellipse, draw = black, fit= (x) (y) (z)] {};
\end{tikzpicture}
\captionsetup{font=footnotesize}
\captionof{figure}{}\label{f1}
\end{minipage}
\begin{minipage}{0.45\textwidth}
\centering
\begin{tikzpicture}[thick,scale=0.6, every node/.style={scale=0.6}]
	\node[vertex, label=$v$] at (0,0) (x1) {} ;
	\node[vsquare, position = -90:{1.2} from x1] (y1) {$h-2$} ;
	\node[position = 0:{1} from y1] (z1) {} ;
	\node[vertex, position = -90:{3.5} from x1] (x2) {} ;
	\node[vsquare, position = -90:{1.8} from x2] (y2) {$h-2$} ;
	\node[vertex, position = -90:{3.5} from x2] (x3) {} ;
	\node[position = -90:{1.2} from x3] (y3) {} ;
	\node[position = -90:{1} from x3] (xm3) {} ;
	\node[position = -90:{5.3} from x3] (xmr) {} ;
	\node[vertex, position = -90:{6.5} from x3] (xr) {} ;
	\node[vsquare, position = -90:{1.2} from xr] (ym) {$h-2$} ;
	\node[position = 0:{1} from ym] (zm) {} ;
	\node[vertex, position = -90:{2.9} from xr, label = $w$] (xm) {} ;
	\node[ellipse, draw = black, fit= (x2) (y1) (x1)] {};
	\node[ellipse, draw = black, fit= (x2) (y2) (x3)] {};
	\node[ellipse, draw = black, fit= (xr) (ym) (xm)] {};
	\draw[dashed] (xm3) -- (xmr) ;
	\draw[dashed] (z1) -- node[right=2pt] {$q$} (zm) ;
\end{tikzpicture}
\captionsetup{font=footnotesize}
\captionof{figure}{}\label{f2}
\end{minipage}

\medskip

Given our representation of the graph $I_q$, we will use a star so that Figures \ref{fstar1} and \ref{fstar2} represent the same graph and a triangle so that Figures \ref{ftri1} and \ref{ftri2} represent the same graphs.
\begin{minipage}{0.45\textwidth}
\centering
\begin{tikzpicture}[thick,scale=0.6, every node/.style={scale=0.6}]
	\node[vertex, label = $v$] at (0,0) (x) {} ;
	\node[star, draw = black, position = -90:{1.6} from x] (y) {} ;
	\node[vertex,position = -90:{4} from x, label = $w$] (z) {} ;
	\node[ellipse, draw = black, fit= (x) (y) (z)] {};
\end{tikzpicture}
\captionsetup{font=footnotesize}
\captionof{figure}{}\label{fstar1}
\end{minipage}
\begin{minipage}{0.45\textwidth}
\centering
\begin{tikzpicture}[thick,scale=0.6, every node/.style={scale=0.6}]
	\node[vertex, label = $v$] at (0,0) (x1) {} ;
	\node[vertex, position = 0:{4} from x1, label = $w$] (x2) {} ;
	\node[position = 0:{2} from x1] (d1) {} ;
	\node[vsquare, position = 90:{1.3} from d1] (h) {$h-2$} {} ;
	\node[vertex, position = 90:{4} from d1] (a1) {} ;
	\node[hex, position = 90:{6} from d1] (a2) {$a$} ;
	\node[vertex, position = 90:{8} from d1] (a3) {} ;
	\node[ellipse, draw = black, fit = (x1) (x2) (a1)] {} ;
	\node[ellipse, draw = black, fit = (a1) (a2) (a3)] {} ;
\end{tikzpicture}
\captionsetup{font=footnotesize}
\captionof{figure}{}\label{fstar2}
\end{minipage}

\begin{minipage}{0.45\textwidth}
\centering
\begin{tikzpicture}[thick,scale=0.6, every node/.style={scale=0.6}]
	\node[vertex, label = $v$] at (0,0) (x) {} ;
	\node[regular polygon, regular polygon sides= 3, draw = black, position = -90:{1.6} from x] (y) {} ;
	\node[vertex,position = -90:{4} from x, label = $w$] (z) {} ;
	\node[ellipse, draw = black, fit= (x) (y) (z)] {};
\end{tikzpicture}
\captionsetup{font=footnotesize}
\captionof{figure}{}\label{ftri1}
\end{minipage}
\begin{minipage}{0.45\textwidth}
\centering
\begin{tikzpicture}[thick,scale=0.6, every node/.style={scale=0.6}]
	\node[vertex, label = $v$] at (0,0) (w1) {} ;
	\node[vertex, label = $w$, position = 90:{4} from w1] (y1) {} ;
	\node[vertex, position = 180:{4} from w1] (r1) {} ;
	\node[vertex, position = 180:{6} from r1] (r2) {} ;
	\node[vertex, position = 180:{4} from y1] (s1) {} ;
	\node[vertex, position = 180:{6} from s1] (s2) {} ;
	\node[hex, position = -90:{2} from r2] (kr3) {$a+1$} ;
	\node[vertex, position = -90:{2}	from kr3] (kr4) {} ;
	\node[ellipse, draw = black, fit = (r2) (kr3) (kr4)] (er2) {} ;
	\node[hex, position = 90:{2} from s2] (ks3) {$a$} ;
	\node[vertex, position = 90:{2}	from ks3] (ks4) {} ;
	\node[ellipse, draw = black, fit = (s2) (ks3) (ks4)] (es2) {} ;
	\node[hex, position = -90:{2} from r1] (kr1) {$a+1$} ;
	\node[vertex, position = -90:{2}	from kr1] (kr2) {} ;
	\node[ellipse, draw = black, fit = (r1) (kr1) (kr2)] (er1) {} ;
	\node[hex, position = 90:{2} from s1] (ks1) {$a$} ;
	\node[vertex, position = 90:{2}	from ks1] (ks2) {} ;
	\node[ellipse, draw = black, fit = (s1) (ks1) (ks2)] (es1) {} ;
	\draw[dashed] (es1) -- node[above=2pt] {$h-b$} (es2) ;
	\draw[dashed] (er1) -- node[above=2pt] {$b+r-2$} (er2) ;
	\node[rounded rectangle, draw = black,fit = (w1) (y1) (s1) (s2) (r2) (r1)] {} ;
\end{tikzpicture}
\captionsetup{font=footnotesize}
\captionof{figure}{}\label{ftri2}
\end{minipage}

\medskip

Given these new pieces of notation, we are now in a position to prove Proposition \ref{counter}

\begin{proof}[Proof of Proposition \ref{counter}]
Let $k \ge 2$ and $d > h \ge 2$, fix $r = d-h$ and suppose that $a \in \mathbb{N}$ and $b \in [h]$ are such that $k=ah+b$. We will have to construct a sequence of bipartite graphs $G_n$ each satisfying the $\alpha_n$-neighbourhood condition with no $(h,hk)$-matching where $\alpha_n$ tends to $\alpha = h-1 + \frac{r+1}{k+1+(r-1)\lceil k/h \rceil} = h-1 + \frac{r+1}{k+1 + (r-1)(a+1)}$. We will do this starting with a small graph $H$ which does not contain a $(h,hk)$-matching and then replacing a copy of $I_{a+1}$ connected to the rest of $H$ through $v_1$ with a large graph $J_n$ in which in every $(h,hk)$-matching, $v_1$ is in a component with at least $h(a+1)$ edges. We give the base graph $H = (U,V,E)$ below.

\begin{minipage}{\textwidth}
\centering
\begin{tikzpicture}[thick,scale=0.6, every node/.style={scale=0.6}]
	\node[label = $v_1$, switch] at (0,0) (x1) {} ;
	\node[vertex, position = -90:{11} from x1] (xm) {} ;
	\node[vertex, position = -90:{7} from x1] (xs) {} ;
	\node[vertex, position = -90:{4} from x1] (xr) {} ;
	\node[hex, position = 0:{1.8} from x1] (y1) {$a+1$} ;
	\node[position = -90:{0.2} from y1] (asdf) {} ;
	\node[vertex, position = 0:{4} from x1] (z1) {} ;
	\node[hex, position = 0:{1.8} from xr] (yr) {$a+1$} ;
	\node[vertex, position = 0:{4} from xr] (zr) {} ;
	\node[hex, position = 0:{1.3} from xm] (ym) {$a$} ;
	\node[vertex, position = 0:{4} from xm] (zm) {} ;
	\node[hex, position = 0:{1.3} from xs] (ys) {$a$} ;
	\node[vertex, position = 0:{4} from xs] (zs) {} ;
	\node[ellipse,draw = black,fit = (x1) (y1) (z1),label=right:$L_1$] (e1) {} ;
	\node[ellipse,draw = black,fit = (xr) (yr) (zr),label=right:$L_{b+r}$] (er) {} ;
	\node[ellipse,draw = black,fit = (xs) (ys) (zs),label=right:$M_1$] (es) {} ;
	\node[ellipse,draw = black,fit = (xm) (ym) (zm),label=right:$M_{h-b}$] (em) {} ;
	\node[ellipse,draw = black,fit = (x1) (xr) (xs) (xm)] {} ;
	\draw[dashed] (e1) -- node[right=2pt] {$b+r$} (er) ;
	\draw[dashed] (es) -- node[right=2pt] {$h-b$} (em) ;
\end{tikzpicture}
\captionsetup{font=footnotesize}

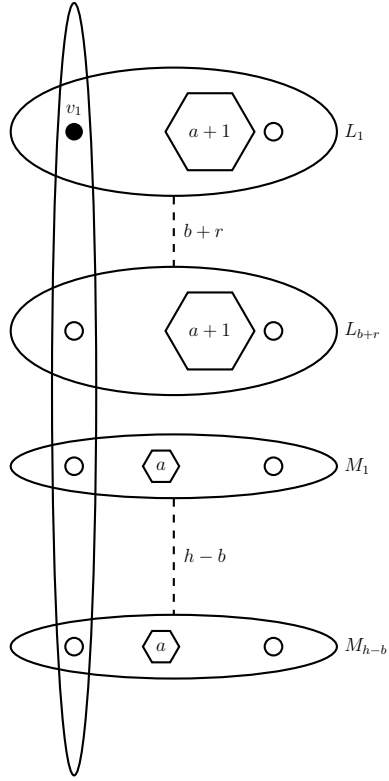
\captionof{figure}{Base Graph $H$}\label{base}
\end{minipage}

\medskip

The augmenting gadget $J_n$ can be thought of an odd cycle where the edges are replaced with copies of the graphs given in figures \ref{fstar1} and \ref{ftri1} alternately with two "star" edges next to each other.

\begin{minipage}{\textwidth}
\centering
\begin{tikzpicture}[thick,scale=0.5, every node/.style={scale=0.5}]
	\node[label = $v_1$, switch] at (0,0) (x1) {} ;
	\node[vertex, position = 90:{4} from x1] (x2) {} ;
	\node[star, draw = black, position = 90:{1.75} from x1] (s2) {} ;
	\node[vertex, position = 90:{4} from x2] (x3) {} ;
	\node[regular polygon, regular polygon sides= 3, draw = black,position = 90:{1.75} from x2] (t3) {} ;
	\node[vertex, position = 180:{6} from x3] (x4) {} ;
	\node[star, draw = black, position = 180:{2.75} from x3] (s4) {} ;
	\node[vertex, position = 270:{4} from x4] (x5) {} ;
	\node[vertex, position = 270:{4} from x5] (x6) {} ;
	\node[regular polygon, regular polygon sides= 3, draw = black,position = 270:{1.75} from x5] (t5) {} ;
	\node[star, draw = black, position = 0:{2.75} from x6] (s6) {} ;
	\draw[dashed] (x4) -- (x5) ;
	\node[ellipse,draw = black, fit = (x1) (s2) (x2),label =right:$S_1$] {} ;
	\node[ellipse,draw = black, fit = (x2) (t3) (x3),label =right:$T_1$] {} ;
	\node[ellipse,draw = black, fit = (x3) (s4) (x4),label = $S_2$] {} ;
	\node[ellipse,draw = black, fit = (x6) (s6) (x1),label =$S_{n+1}$] {} ;
	\node[ellipse,draw = black, fit = (x5) (t5) (x6),label = right:$T_n$] {} ;
\end{tikzpicture}
\captionsetup{font=footnotesize}

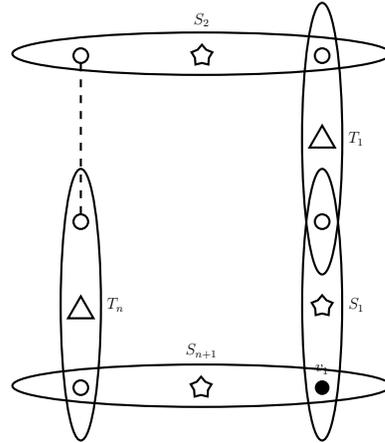
\captionof{figure}{Augmenting Gadget $J_n$}\label{fn}
\end{minipage}

\medskip

To form $G_n$, we first remove $L_1 \setminus \{v_1\}$ from $H$ to give $H'$. We will then identify the vertices labelled $v_1$ in $H'$ and the augmenting gadget $J_n$ to form $G_n = (U_n,V_n,E_n)$. To prove our proposition, it suffices to show that $G_n$ does not contain a $(h,hk)$-matching and that it satisfies the $\alpha_n$-neighbourhood condition where $\alpha_n$ increases to $\alpha$ as $n$ tends to infinity.

So suppose that $G_n$ contains a $(h,hk)$-matching. This induces a $(h,hk)$-matching on $J_n$. One can verify that in any $(h,hk)$-matching on $J_n$ that $v_1$ is in a component with at least $(a+1)h$ edges. Then the matching on $G_n$ must induce a $(h,hk)$-matching on $H'$ where $v$ is a component with at most $(k-(a+1))h$ edges. This is a contradiction though, since we could extend this to a $(h,hk)$-matching on $H$ by adding back $L_1$. So $G_n$ cannot contain a $(h,hk)$-matching.

It is also easy to verify that $R_{G_n}(S)$ is minimised over $S \subseteq U_n$ when $U_n$. $G_n$ then satisfies the $R_{G_n}(U_n)$-neighbourhood condition where
	\beqs
		R_{G_n}(U_n) &=& \frac{|V_n|}{|U_n|} \nonumber \\
		&=& \frac{|V| - (a+1)(h-1) + |B_n|-1}{|U|-(a+1)+|A_n|}. \nonumber
	\eeqs
	
After some simplification this becomes
	\beqs
		R_{G_n}(U_n) &=& \frac{(h-1)(a(h+r)+b+r+1) + r+1+ n(r+1) + n(a(h+r-1)+b+r)}{1+a(h+r)+b+r) + n(a(h+r-1)+b+r)} \nonumber \\
		&=& h-1 + \frac{n(r+1) + r+1}{1+a(h+r) + b+r + n(a(h+r-1)+b+r)} \nonumber \\
		&=& h-1 + \frac{(n+1)(r+1)}{(n+1)(ah+b + ar - a +r)+1} \nonumber \\
		&=& h-1 + \frac{(n+1)(r+1)}{(n+1)(k+1 + (a+1)(r-1))+1}. \nonumber
	\eeqs
We now see that $R_{G_n}(U_n)$ tends to $\alpha$ as $n$ tends to infinity and so we are done.

\end{proof}

\section{$k$-star covering}\label{starsect}
A naturally related problem is the following: Under what conditions can you cover a graph with trees of bounded size?. It is clear that only stars will be necessary since, for any tree with diameter at least three contains an edge between two non-leaf vertices which we may delete.

\begin{defn}
Let $k \ge 1$ be a positive integer and $G = (U,V,E)$ be a bipartite graph. A \em $k$-star covering \em is a subset $F$ of $E$ such that in $H = (U,V,F)$, each component is a star with at most $k$ edges in it, and $d_H(x) \ge 1$, for each $x \in U \cup V$.
\end{defn}

The following result gives a necessary and sufficient condition for a bipartite graph to have a $k$-star covering.

\begin{thm}\label{stars}
Let $G = (U,V,E)$ be a bipartite graph. Then $G$ has a $k$-star covering iff $|\Gamma(S)| \ge \frac{1}{k}|S|$ for all $S \subset U$ and $S \subset V$.
\end{thm}

Note that we require all vertices to be covered in a star covering and so the above is not equivalent to the existence of a $(1,k)$-matching (which may only cover the left vertices).

\begin{proof}[Proof of Theorem \ref{stars}]
For a bipartite graph $G = (U,V,E)$, we shall say that $G$ satisfies the \em double-sided $\alpha$-neighbourhood condition \em if $|\Gamma(S)| \ge \alpha |S|$ for any $S \subset U \cup V$. Theorem \ref{stars} can be reformulated as $G = (U,V,E)$ contains a $k$-star covering if and only if it satisfies the double-sided $\frac{1}{k}$-neighbourhood condition.

The necessity of the double-sided neighbourhood condition can be seen by counting edges. Suppose that $G = (U,V,E)$ is a bipartite graph with a $k$-star covering $F \subset E$. If we let $H = (U,V,F)$, and $S \subset U$, then $|S| \le e_H(S,\Gamma(S)) \le e_H(U,\Gamma(S)) \le k|\Gamma(S)|$, since each vertex $v$ in $V$ has degree $d_H(v) \le k$ and similarly if $S \subset V$, $|S| \le k|\Gamma(S)|$. Therefore, $G$ satisfies the doubled-sided $\frac{1}{k}$-neighbourhood condition.

It remains to show sufficiency. Let $G=(U,V,E)$ be a bipartite graph such that $|\Gamma(S)| \ge \frac{1}{k}|S|$ for each $S \subset U$ and $S \subset V$ and suppose that $G$ is minimal with respect to $|E|$ such that it does not have a $k$-star covering. $G$ is minimal with respect to edges so it must be connected and for all $uv \in E$, where $u \in U$ and $v \in V$, it must be the case that there is either a set $S \subset U$ such that $u \in S$, $v \notin \Gamma(S \setminus u)$ and $h(U,\frac{1}{k}) < 1$, or a set $T \subset V$ such that $v \in T$, $u \notin \Gamma(T \setminus \{v\})$ and $h(U,\frac{1}{k})<1$. Now suppose that there exists an edge $uv \in E$ such that $d(u),d(v) \ge 2$ and assume without loss of generality that $S \subset U$ is such that $u \in S$, $v \notin \Gamma(S \setminus u)$ and $h(S,\frac{1}{k}) <1$. Note that since $h(S,\frac{1}{k})$ must be a multiple of $\frac{1}{k}$, $h(S,\frac{1}{k}) \le \frac{k-1}{k}$. Further note that $h(\{u\},\frac{1}{k}) = d(u) - \frac{1}{k} \ge 1$ and so $S \setminus u \neq \emptyset$, and
	\beqs
		h(S \setminus u, \frac{1}{k}) &=&  |\Gamma(S \setminus u)| - \frac{1}{k}|S \setminus u| \nonumber \\
		&=& |\Gamma(S)| - |\Gamma(u) \setminus \Gamma(S)| - \frac{1}{k}(|S|-1) \nonumber \\
		&=& h(S,\frac{1}{k}) + \frac{1}{k} - |\Gamma(u) \setminus \Gamma(S)|. \label{covercount1}
	\eeqs
Note that $h(S,\frac{1}{k}) \le \frac{k-1}{k}$ and $v \in \Gamma(u) \setminus \Gamma(S)$, so $|\Gamma(u) \setminus \Gamma(S)| \ge 1$. Putting these into \eqref{covercount1} gives us that $h(S \setminus u,\frac{1}{k}) \le 0$ and so since $h$ is positive, $h(S \setminus u,\frac{1}{k}) = 0$. Consider $H = G[(S \setminus u) \cup \Gamma(S \setminus u)]$. $|\Gamma_H(A)| \ge \frac{1}{k}|A|$ for all $A \subset (S \setminus u)$ and so there must be a $(1,k)$-matching on $H$ (consider blowing up the right vertices and applying Hall's Theorem). However, since $|\Gamma(S \setminus u)| = \frac{1}{k}|S \setminus u|$, each vertex in $\Gamma(S \setminus u)$ must be used in this covering and so  this $(1,k)$-matching must in fact be a $k$-star covering. It must then be the case that $J = G[(U \setminus (S \setminus u)) \cup (V \setminus \Gamma(S \setminus u))]$ cannot have a $k$-star covering, else we may take the union of the vertex disjoint $k$-star coverings of $H$ and $J$ to get a $k$-star covering of $G$. On the other hand, if $A \subset (U \setminus (S \setminus u))$, then 
	\beqs
		|\Gamma_J(A)| &\ge& |\Gamma_G(A \cup (S \setminus u))| - |\Gamma_G(S \setminus u)| \nonumber \\
		&=& |\Gamma_G(A \cup (S \setminus u))| - \frac{1}{k}|S \setminus u| \nonumber \\
		&\ge& \frac{1}{k}|A \cup (S \setminus u)| - \frac{1}{k}|S \setminus u| \nonumber \\
		&\ge& \frac{1}{k}|A|. \nonumber
	\eeqs
If $B \subset V \setminus(\Gamma(S \setminus u))$, then $\Gamma_J(B) = \Gamma_G(B)$ and so $|\Gamma_J(B)| \ge \frac{1}{k}|B|$. Therefore $J$ satisfies the double-sided $\frac{1}{k}$-neighbourhood condition but does not have a $k$-star covering. This contradicts the edge-minimality of $G$ and so there can be no such edge $uv$ with $d(u),d(v) \ge 2$.

So it must be the case that each edge in $G$ must be incident to a leaf. The only such connected graphs are stars and so $G$ must be a star. $G$ would then already be a $k$-star covering and so we arrive at a contradiction.
\end{proof}

It would be interesting to find a "Tutte-style" result for the existence of a $k$-star covering in a general (non-bipartite) graph.

\end{document}